\documentclass[final]{siamltex}
\usepackage[compress]{cite}
\usepackage{amssymb}
\usepackage{subfigure}
\usepackage[colorlinks,linkcolor=blue,citecolor=green,urlcolor=red]{hyperref}
 \usepackage{graphicx}
\usepackage{amsmath,bm}
  \numberwithin{figure}{section}
  \numberwithin{table}{section}
\numberwithin{hypothesis}{section}
\newtheorem{remark}{Remark}
\numberwithin{remark}{section}
\numberwithin{theorem}{section}

\newtheorem{assumption}{Assumption}
\makeatletter
\renewcommand\normalsize{%
   \@setfontsize\normalsize\@xpt\@xiipt
   \abovedisplayskip 8.5\p@ \@plus2\p@ \@minus5\p@
   \abovedisplayshortskip \z@ \@plus3\p@
   \belowdisplayshortskip 8.5\p@ \@plus3\p@ \@minus3\p@
   \belowdisplayskip \abovedisplayskip
   \let\@listi\@listI}
\makeatother
\title{Stability Implies Robust Convergence of a Class of Preconditioned Parallel-in-time Iterative  Algorithms}
\author{Shu-Lin Wu\thanks{School of Mathematics and Statistics, Northeast Normal University, Changchun 130024, China.
E-mail: \texttt{wushulin84@hotmail.com}} 
\and Tao Zhou\thanks{LSEC, Institute of Computational Mathematics and Scientific/Engineering Computing, 
AMSS, Chinese Academy of Sciences, Beijing, 100190, China. E-mail: \texttt{tzhou@lsec.cc.ac.cn}}
\and  Zhi Zhou\thanks{Department  of  Applied  Mathematics,  The  Hong  Kong  Polytechnic  University,  Kowloon,Hong Kong.
E-mail: \texttt{zhizhou@polyu.edu.hk}}}
 
\begin{document}

\maketitle

\begin{abstract}
Solving evolutionary equations in a parallel-in-time manner is an attractive topic and many algorithms are proposed in recent two decades. The algorithm based on the block $\alpha$-circulant preconditioning technique has shown promising  advantages,  especially for wave propagation problems. By fast Fourier transform for factorizing the involved circulant matrices, the preconditioned iteration can be computed efficiently via the so-called diagonalization technique, which yields a direct parallel implementation across all time levels. In recent years, considerable efforts have been devoted to exploring  the convergence of the preconditioned iteration 
by studying the spectral radius of the iteration matrix, and this leads to many case-by-case studies depending on the used time-integrator. In this paper,  we propose a unified convergence analysis for the algorithm applied to $u'+Au=f$,  where   $\sigma(A)\subset\mathbb{C}^+$ with $\sigma(A)$ being  the spectrum of $A\in\mathbb{C}^{m\times m}$. For any   one-step method (such as the Runge-Kutta methods) with stability function $\mathcal{R}(z)$,  we prove  that the decay rate of the global error is bounded by $\alpha/(1-\alpha)$, provided the  method is stable, i.e., $\max_{\lambda\in\sigma(A)}|\mathcal{R}(\Delta t\lambda)|\leq1$. For any linear multistep method, such a bound becomes $c\alpha/(1-c\alpha)$, where $c\geq1$ is a constant specified by the multistep method itself. 
Our proof only relies on  the stability of the time-integrator and the estimate is independent of the step size $\Delta t$ and  the   spectrum $\sigma(A)$.  
\end{abstract}



\pagestyle{myheadings}
\thispagestyle{plain}
%\markboth{ J. Liu and S. L. Wu}{Block $\alpha$-circulant preconditioner  for wave equation}
%Dear Professor Li,

%We are working for solving wave equations in a parallel-in-time pattern
\section{Introduction}\label{sec1}
% In this paper, we are particularly interested in evolutionary partial differential equations (PDEs) with first order temporal derivative. The classical time-stepping method solves the evolutionary PDEs one time level by one time level (i.e., in a fully sequential manner), which {\color{red}would be}  time-consuming for {\color{red}long time computation and/or it is expensive to handle the algebraic system at each step}. This motivates the development of {\color{red}parallel-in-time (PinT)} methods for evolutionary PDEs {\color{red}during the last two decades. Among these, we mention the \textit{ parareal} algorithm [13] and a closely related algorithm   MGRiT  [5], which attract  considerable attention in recent years.    Convergence of parareal and MGRiT are respectively  justified in [8]  and {\color{blue}[Ref1]}. Many efforts are devoted to  improving  these two PinT algorithms and in particular the authors  {\color{blue}[Ref2]} and {\color{blue}[Ref3]} proposed a novel  coarse grid correction, which shows great potential for increasing the speedup according to the numerical results in {\color{blue}[Ref4]}.  There are also many another PinT algorithms with  completely different mechanism from parareal and MGRiT, such as the space-time multigrid algorithms [7,10,11] and the diagonalization-based all-at-once algorithms [9,14,15,18,{\color{blue}Ref5}].   For an  overview,  we refer the interested reader to [6].}
We are interested in solving  the following evolutionary equation parallel-in-time (PinT):
\begin{equation}\label{eq1.1}
y'+Ay=g, ~ A\in\mathbb{C}^{m\times m}, 
\end{equation}
with initial value $y(0)=y_0$. The spectrum of the matrix $A$ is supposed to lay on the right hand-side of the complex plane, i.e.,    
$$\sigma(A)\subseteq \mathbb{C}^+ =\{z=\delta+i\omega: \delta\geq0, \omega\in\mathbb{R}\}.$$ 
Note that the system \eqref{eq1.1} is widely used as an spatially semi-discrete scheme
of many evolutionary partial differential equations, e.g.:
\begin{itemize}
\item Diffusion equation: $u_t - \Delta u  = f$,where $A=-\Delta_h \approx\Delta$ denotes the discrete Laplacian, 
and we have $\sigma(A)\subseteq\mathbb{R}^+$.
\item Acoustic wave equation:  $u_{tt}-\Delta u=f$. In this case,
we can  make an order reduction to yield  the first-order system \eqref{eq1.1} with    
\begin{equation}\label{eq1.2}
 y=\begin{bmatrix}u\\ v\end{bmatrix}, ~A=\begin{bmatrix} &-I_x\\ -\Delta_h &\end{bmatrix},~g=\begin{bmatrix}0\\ f\end{bmatrix},
\end{equation}
Then we have $\sigma(A)\subseteq{\rm i}\mathbb{R}^+$.
\item Schrodinger equation: ${\rm i} u_t-\Delta u+V(x)u=f$ with a positive potential term $V(x) > 0$. Then we have $\sigma(A)\subseteq{\rm i}\mathbb{R}^+$.
\end{itemize}

In the recent two decades, the research toward reducing the computation time 
via  time parallelization is a hot topic (see  \cite{Gander:2015} for a comprehensive review). 
In this paper, we will focus on the approach
that uses the diagonalization technique proposed by Maday and Ronquist in \cite{MR08}.
In that pioneer work, they formulated the time-stepping system %\eqref{eqn:singe-step} 
into a space-time \textit{all-at-once} system. Then they
diagonalize the time stepping matrix and solve all time steps in parallel. 
Such a technique  was first used in \cite{GH19} for wave equations.  
The concrete algorithm  lies in using the geometrically increasing time step-sizes
$\{\Delta t_n =\mu^{n-1}\Delta t_1\}_{n=1}^{N_t}$  with some parameter  $\mu>1$, in order 
to make the time discretization matrix diagonalizable. The  algorithm is 
directly parallel, but the roundoff error  arising from the diagonalization 
procedure increases dramatically for large $N_t$.
In order to balance the roundoff error and the discretization error,
 the quantity $N_t$ can not be too large (in practice, people choose $N_t= 20\sim25$).
 
To overcome the restriction on $N_t$, the diagonalization technique was used in an iterative fashion 
\cite{MPW18,LW20,GZ20,LN18,LN20,GW19b}, that will be investigated in this paper. 
The idea is briefly explained as follows. First,   we   partition  the time domain $(0,T)$  by  $N_t$  equally spaced time points 
 $\{t_n\}_{n=0}^{N_t}$ with 
step size  $\Delta t=T/N_t$. Then  we apply a time-integrator to \eqref{eq1.1} (such as Runge-Kutta (RK) methods)
\begin{equation}\label{eqn:single-step}
 y_{n+1}+\mathcal{R}(\Delta tA)y_n=\eta_n, \quad n=0,1,\dots, N_t-1,
\end{equation}
with a given initial value  $y_0$,  where  $\eta_n$ is {a known quantity specified by} the source term at $t_n$. The increment matrix  $\mathcal{R}(\Delta tA)$ 
is  deduced from the stability function of  the time-integrator.   Instead of solving the $N_t$ difference equations \eqref{eqn:single-step} level-by-level,  we form these equations into   an \textsl{all-at-once} system
\begin{equation}\label{eqn:all-at-once}
 \mathcal{K}{\bm u}={\bm b},
 \end{equation}
where ${\bm u}=(y_1^\top, y_2^\top, \dots, y_{N_t}^\top)^\top$, ${\bm b}=((\eta_0-\mathcal{R}(\Delta tA)y_0)^\top, \eta_1^\top,\dots, \eta_{N_t}^\top)^\top$ and 
 \begin{equation}\label{eqn:K}
\mathcal{K}=I_t\otimes I_x+B\otimes \mathcal{R}(\Delta tA). 
 \end{equation}
 Here and hereafter  $I_t\in \mathbb{R}^{N_t\times N_t}$ and $I_x\in \mathbb{R}^{m\times m}$ 
 denotes the identity matrices in time and space respectively and the  Toeplitz matrix $B$ is
\begin{equation}\label{eqn:C}
B=\begin{bmatrix}
0 & & &\\
1&0 & &\\
&\ddots &\ddots &\\
& &1 &0
\end{bmatrix}\in\mathbb{R}^{N_t\times N_t}. 
 \end{equation}
To solve the all-at-once system \eqref{eqn:all-at-once}, we propose a block $\alpha$-circulant preconditioner: 
\begin{equation}\label{eqn:P}
\mathcal{P}_\alpha=I_t\otimes I_x+C(\alpha)\otimes \mathcal{R}(\Delta tA), 
 \end{equation}
 where $C(\alpha)$ denotes the following $\alpha$-circulant matrix
 \begin{equation}\label{eqn:C-alpha}
  C(\alpha)=\begin{bmatrix}
0 & & &\alpha\\
1&0 & &\\
&\ddots &\ddots &\\
& &1 &0
\end{bmatrix}\in\mathbb{R}^{N_t\times N_t},
 \end{equation}
and apply the preconditioned iteration:
\begin{equation}\label{eqn:precond-iteration}
\mathcal{P}_\alpha\Delta{\bm u}^k={\bm r}^k,  {\bm u}^{k+1}={\bm u}^{k}+\Delta {\bm u}^k, ~{\bm r}^k:={\bm b}-\mathcal{K}{\bm u}^{k}, ~k=0, 1,\dots.
\end{equation}
 Thanks to the property of the $\alpha$-circulant matrices (see, e.g., \cite[Lemma 2.1]{BLM05}), 
 we  can make a block Fourier diagonalization for  the preconditioner $\mathcal{P}_\alpha$    and this allows to compute
 $\mathcal{P}_\alpha^{-1}(\mathcal{K}{\bm u}^{k}-{\bm b})$  in \eqref{eqn:precond-iteration} highly parallel for
 all time levels  \cite{GZ20,LN18,LN20,LW20,MPW18,GW19b}.
 
In this paper, we aim to answer the question:  \textit{under what conditions 
the iterative algorithm \eqref{eqn:precond-iteration} converges rapidly and robustly?} 
In existing references, people answered this question by examining the radius 
of the spectrum of the iteration matrix $\mathcal{I}-\mathcal{P}_\alpha^{-1}\mathcal{K}$, see e.g., 
 \cite{MPW18,GW19b,LN18,LN20,WZ21}  for
 parabolic problems and \cite{GW19a,LW20}  for the wave equations.
For some special
  time-integrators (e.g., the implicit Euler method \cite{LN18,LN20}, 
  {the implicit leap-frog method \cite{LW20}},
  {and the two-stage singly diagonal implicit RK  method \cite{WZ21}})  it was shown 
   \begin{equation}\label{eqn:spectrum}
\rho(\mathcal{I}-\mathcal{P}_\alpha^{-1}\mathcal{K})\leq\frac{\alpha}{1-\alpha}, ~\alpha\in (0, 1).
\end{equation}
The analysis  in these   works is rather technical and  {heavily} depends on the  special   
property of the time-integrator, e.g.,  the sparseity, Toeplitz structure  and diagonal dominance 
of the time-discretization matrix.% \cite{LN18,LN20,LW20,MPW18,WZ21,GZ20}.

Instead of exploring the spectrum of the iterative matrix,
we shall examine the error of the preconditioned iteration directly, and prove that for \textit{any}  one-step  time stepping method  \eqref{eqn:single-step}
  the error  ${\bm{err}}^{k}={\bm u}^k-{\bm u}$ of the algorithm \eqref{eqn:precond-iteration} satisfies (Theorem \ref{thm:single})
\begin{equation}\label{eq1.11} 
\|(I_t\otimes P){\bm{err}}^{k+1}\|_{\infty}\leq\frac{\alpha}{1-\alpha}\|(I_t\otimes P){\bm{err}}^{k}\|_{\infty},
\end{equation} 
where $P$ 
is the eigenvector matrix of $A$,  provided the method solving \eqref{eq1.1}  is stable, i.e., 
$\max_{\lambda\in\sigma(A)}|\mathcal{R}(\Delta t\lambda)|\leq1$.  If $\alpha\in(0,1/2)$, the precondtioned iteration  \eqref{eqn:precond-iteration}
converges linearly with the convergence factor $\alpha/(1-\alpha)$.
The  result \eqref{eq1.11}  well explains the robust convergence observed in  \cite{LN18,LN20,GW20,GW19a}, 
because all the time-integrators used there are unconditionally stable.  

Similarly, for \textit{any} stable linear multistep method, we can also develop a preconditioned iterative solver,
and show that the iteration converges linearly (for  {suitable parameter} $\alpha$) and satisfies (Theorem \ref{thm:multi})
\begin{equation}\label{eq1.12} 
\|(I_t\otimes P){\bm{err}}^{k+1}\|_{\infty}\leq\frac{c\alpha}{1-c\alpha}\|(I_t\otimes P){\bm{err}}^{k}\|_{\infty},
\end{equation} 
where $c\ge1$ is a generic constant only depending on the stability property of the multistep method.  
For parabolic equations and subdiffusion models with memory effects, 
such an iterative algorithm for $r$-step backward differentiation formula up to $r=6$ has 
been analyzed in a most recent work \cite{WuZhou:BDF}. 
The current work provides a more systematic argument showing the relation between
stability of time-integrators and the convergence of the iterative solver \eqref{eqn:precond-iteration}.

%This result  covers a most recent work in   \cite{WuZhou:BDF}  as a special case, where the authors made  an analysis 
%for the $r$-step backward differentiation formula (BDF) up to $r=6$ for
%solving parabolic equations and nonlocal-in-time subdiffusion equations.

%Note that the convergence of a PinT algorithm for parabolic equations was analyzed
%in \cite{WuZhou:BDF} via applying the stability of BDF schemes.
%The current work provides a more systematic argument showing the relation between
%stability of time-integrators and the convergence of the iterative solver \eqref{eqn:precond-iteration}.
 
These estimates imply that the convergence of the
iterative algorithm \eqref{eqn:precond-iteration} is very fast when we choose a very small $\alpha$. 
However, in practice, the parameter $\alpha$ can not be arbitrarily small,  because the roundoff error 
arising  {from the} block  diagonalization of $\mathcal{P}_\alpha$ 
will increase dramatically when $\alpha\rightarrow 0$. Such a roundoff error is explained as follows. For any diagonalizable square matrix
$Q$ with eigenvalue matrix $D$ and eigenvector matrix $V$, due to the floating point operations $Q$ can not precisely equal to
$VDV^{-1}$ and it only holds  $Q\approx VDV^{-1}$. In our case, for the $\alpha$-circulant matrix $C(\alpha)$ in \eqref{eqn:C-alpha}, 
%even though the eigenvalues and eigenvector can be mathematically expressed, 
the difference between $C(\alpha)$ and $VDV^{-1}$  
continuously increase as $\alpha$ decreases. See \cite{GW19a, WuZhou:BDF} for some detailed studies 
on the roundoff error of the diagonalization 
procedure, where the authors proved
\begin{equation}\label{eq1.13}
\texttt{roundoff error}=O(\epsilon\alpha^{-2}),
\end{equation}
 where  $\epsilon$  is the machine precision ($\epsilon= 2.2204\times10^{-16}$ for a 32-bit computer).
 
{The rest of this paper is organized as follows. In Section \ref{sec2}, we prove \eqref{eq1.11} for the one-step time-integrators. The proof  lies in representing  the preconditioned  iteration \eqref{eqn:precond-iteration} as a special difference equations with head-tail coupled condition and    relays on the stability function only.  We also consider the  linear multistep methods in Section \ref{sec:multistep}, for which  we prove    \eqref{eq1.12}  and  the stability  of the numerical method plays a central role in the proof as well. We conclude this paper in Section \ref{sec4} by listing some important issues that need to be addressed in the future.  }

 \section{Convergence of iteration \eqref{eqn:precond-iteration} for  {one-step} methods}\label{sec2}
In this section, we  prove the convergence of the iterative solver \eqref{eqn:precond-iteration}.
To this end,  {we  assume} that the  {one-step} time-integrator \eqref{eqn:single-step} is stable  {in the following sense}.
\begin{assumption}\label{ass1}
The matrix $A$ in the linear system \eqref{eq1.1} is diagonalizable as $A=PD_AP^{-1}$  { with   $\sigma(A)\subset \mathbb{C}^+$. For such a system, it holds}   
 \begin{equation}\label{eq2.2}
 |\mathcal{R}(\Delta t \lambda)|\leq1  \qquad {\forall ~\lambda\in\sigma(A)}, 
 \end{equation}
 where $\mathcal{R}(\cdot)$ denotes the stability function of the {one-step} method \eqref{eqn:single-step}. 
\end{assumption}
 
Then we are ready to state our main theorem.

\begin{theorem}\label{thm:single}
 {Under  Assumption 1,   the} error at the $k$-th iteration  {in} \eqref{eqn:precond-iteration}, denoted by ${\bm{err}}^{k}={\bm u}^{k}-{\bm u}$,
satisfies 
\begin{equation}\label{eqn:err-single}
\|(I_t\otimes P){\bm{err}}^{k+1}\|_{\infty}\leq\frac{\alpha}{1-\alpha}\|(I_t\otimes P){\bm{err}}^{k}\|_{\infty}\qquad \forall~ k\geq1.
\end{equation}
%provided the time-integrator \eqref{eq2.2} is stable.%, i.e., $\max_{z\in\sigma(\Delta tQ)}|\mathcal{R}(z)|\leq1$,
%where $\sigma(\Delta tQ)=\cup_{j=1,2,\dots, N_x}\{-{\rm i}\Delta t\sqrt{\lambda_j(A)}, {\rm i}\Delta t\sqrt{\lambda_j(A)}\}$ denotes the spectrum of $\Delta tQ$.
Therefore, the iteration \eqref{eqn:precond-iteration} converges linearly if $\alpha\in(0,1/2)$.
\end{theorem}

\begin{proof}
 From \eqref{eqn:precond-iteration}, 
 we observe that the error ${\bm{err}}^{k}$ satisfies
 \begin{equation*}
{\bm{err}}^{k+1}={\bm{err}}^{k}-(\mathcal{P}_\alpha^{-1}\mathcal{K}){\bm{err}}^{k}, ~k=0, 1,\dots.
 \end{equation*}
Now we  define   
$$
{\bm \zeta}^{k}=(({\bm\xi}^k_{1})^\top, ({\bm \xi}^k_{2})^\top, \dots, ({\bm \xi}^k_{m})^\top)^\top=(I_t\otimes P){\bm{err}}^{k}.
$$ 
Let  $z$ be an arbitrary eigenvalue of $\Delta t A$  and  ${\bm \xi}^{k}\in\mathbb{R}^{N_t}$ be the corresponding subvector of ${\bm \zeta}^{k}$.   Then, it is clear that
   \begin{equation}\label{eq2.3}
{\bm \xi}^{k+1}={\bm \xi}^{k}-(P_\alpha^{-1}(z)K(z)){\bm \xi}^{k}, ~k=0, 1,\dots.
 \end{equation}
 with
    \begin{equation*} 
 K(z)=I_t+\mathcal{R}(z)C_t, \quad P_\alpha(z)=I_t+\mathcal{R}(z)C(\alpha).
 \end{equation*}
By applying $\mathcal{P}_\alpha$ on \eqref{eq2.3}, we derive that for $k\ge1$
   \begin{equation}\label{eq2.3-1}
   \begin{split}
P_\alpha(z) {\bm \xi}^{k+1}=(P_\alpha(z) -K(z)){\bm \xi}^{k}.
\end{split}
 \end{equation}
 {Let}
$${\bm \xi}^{k}=(\xi^{k}_{1}, \xi^{k}_{2}, \dots, \xi^{k}_{N_t})^\top.$$
Then the relation \eqref{eq2.3-1} implies
%We claim that   \eqref{eq2.8} is equivalent to the following system
% \begin{equation}\label{eq2.9}
% \begin{cases}
%\xi^{k+1}_{n+1}+\mathcal{R}(z)\xi^{k+1}_{n}=0,  n=0,2,\dots, N_t-1,\\
%\xi^{k+1}_{0}=\alpha(\xi^{k+1}_{N_t}-\xi^{k}_{N_t}),
%\end{cases}
%\end{equation}
%where  $\xi^{k}_{N_t}\neq0$ for $k=0$   (otherwise for all $k\geq1$   the iterates are zero).
\begin{equation*} 
{\begin{split}
\left(I_t +\mathcal{R}(z)C_t\right) {\bm \xi}^{k+1} = \alpha\begin{bmatrix}
 \mathcal{R}(z) (\xi^{k+1}_{N_t} - \xi^k_{N_t})\\0\\
\vdots\\
0
\end{bmatrix}.
\end{split}}
\end{equation*}
Therefore,  \eqref{eq2.3} is equivalent to the following  {difference equations with a head-tail coupled condition} 
 \begin{equation}\label{eq2.5}
 \begin{cases}
\xi^{k+1}_{n+1}+\mathcal{R}(z)\xi^{k+1}_{n}=0, ~ n=0,1,\dots, N_t-1,\\
\xi^{k+1}_{0}=\alpha(\xi^{k+1}_{N_t}-\xi^{k}_{N_t}).
\end{cases}
\end{equation}
 
By Assumption 1, we know that the time-integrator is stable, i.e., $|\mathcal{R}(z)|\leq1$.
This together with the first relation in \eqref{eq2.5} immediately implies  for all $k\geq0$
$$
\left|\xi^{k+1}_{n+1}\right|=|\mathcal{R}(z)|\left|\xi^{k+1}_{n}\right|\leq  \left|\xi^{k+1}_{n}\right|.
$$ 
Hence, the the second relation in \eqref{eq2.5}  leads to 
$$
\left|\xi^{k+1}_{n}\right|\leq \left|\xi^{k+1}_{0}\right|=\alpha\left|\xi^{k+1}_{N_t}-\xi^{k}_{N_t}
\right|\leq\alpha\left|\xi^{k+1}_{N_t}\right|+\alpha\left|\xi^{k}_{N_t}\right|\leq \alpha\left|\xi^{k+1}_{n}\right|+\alpha\left|\xi^{k}_{n}\right|,
$$
where $n=0, 1,\dots, N_t$.  That is,    
$$\left|\xi^{k+1}_{n}\right|\leq\frac{\alpha}{1-\alpha}\left|\xi^{k}_{n}\right|.$$
To summarize, it holds 
\begin{equation} \label{eq2.13}
\|{\bm \xi}^{k+1}\|_{\infty}\leq\frac{\alpha}{1-\alpha}\|{\bm \xi}^{k}\|_{\infty}.
\end{equation}
For ${\bm \zeta}^{k}=(({\bm\xi}^k_{1})^\top, ({\bm \xi}^k_{2})^\top, \dots, ({\bm \xi}^k_{m})^\top)^\top$ with any index $k$, it is clear that
$$
\|{\bm \zeta}^k\|_{\infty}=\max_{j=1,2,\dots,m}\|{\bm \xi}_j^k\|_{\infty}.
$$
Since ${\bm \xi}^{k}$ is an arbitrary subvector of ${\bm \zeta}^{k}$,
it follows from \eqref{eq2.13}  that
$\|{\bm \zeta}^{k+1}\|_{\infty}\leq\frac{\alpha}{1-\alpha}\|{\bm \zeta}^{k}\|_{\infty}$. This, 
together with ${\bm \zeta}^{k}=(I_t\otimes P){\bm{err}}^{k}$, gives the desired result.
 \end{proof}
 
\begin{remark}
Theorem \ref{thm:single} reveals the relation between the stability 
of the time stepping scheme \eqref{eqn:single-step} and the convergence
of preconditioned iterative solver \eqref{eqn:precond-iteration}.
The proof only utilizes the stability of the time stepping scheme
instead of any matrix structure of the all-at-once system explicitly,
and can be extended to other numerical schemes,
e.g. linear multistep methods, which will be introduced in the
next section.
\end{remark}

To test the sharpness of  the estimate \eqref{eqn:err-single}, we consider the following   advection-diffusion equation  
\begin{equation}\label{eq2.7}
\begin{cases}
u_t-\nu u_{xx}+u_x=0, &(x,t)\in(-\frac{1}{2}, \frac{1}{2})\times (0, T),\\
u(x,0)=\sin(2\pi x), &x\in(-\frac{1}{2}, \frac{1}{2}),%\\
%u(-\frac{1}{2},t)=u(\frac{1}{2}, t), &t\in(0, T),
\end{cases}
\end{equation}
with the periodic boundary condition,
where $\nu>0$ is a constant and for $\nu$ small the solution $u$ has obvious characteristic of wave propagation.     We use the  centered finite difference formula to discretize the spatial derivates    
$$
u_{xx}(x_i, t)\approx\frac{u_{i+1}(t)-2u_{i}(t)+u_{i-1}(t)}{\Delta x^2}\quad\text{and}\quad 
u_{x}(x_i, t)\approx\frac{u_{i+1}(t)-u_{i-1}(t)}{2\Delta x},
$$  
with  $i=1,2,\dots, N_x-1$ and   $\Delta x=\frac{1}{N_x}$. This leads to ODE system \eqref{eq1.1} with  
\begin{equation}\label{eq2.8}
\begin{split}
A=\frac{\nu}{\Delta x^2}
\begin{bmatrix} 2 &-1 & & &-1\\
-1 &2 &-1  & &\\
&\ddots &\ddots &\ddots &\\
& &-1 &2 &-1\\
-1 & & &-1 &2
\end{bmatrix}+
\frac{1}{2\Delta x}
\begin{bmatrix}  0&-1 & & &1\\
1 &0 &-1  & &\\
&\ddots &\ddots   &\ddots &\\
& &1 &0 &-1\\
-1 & & &1 &0
\end{bmatrix}. 
\end{split}
\end{equation}
  For time-discretization, we consider the following  two-stage SDIRK method 
 \begin{equation}\label{eq2.9}
\begin{array}{r|cc}
\gamma &\gamma &0     \\
\gamma+\tilde{\gamma}&\tilde{\gamma}   &\gamma\\
\hline
  &b&1-b
  \end{array}, 
\end{equation}
where $\gamma>0$ and $b\neq 1$. 
From  \cite[Chapter IV.6]{HW00},    the method is of  order 2  if 
\begin{subequations}
\begin{equation}\label{eq2.10a}
\gamma b+(\tilde{\gamma}+\gamma)(1-b)=\frac{1}{2}\left(\Rightarrow\tilde{\gamma}=\frac{\frac{1}{2}-\gamma b}{1-b}-\gamma\right).
\end{equation}
 The stability function is 
\begin{equation}\label{eq2.10b}
 \mathcal{R}(z)= 
\frac{(2\gamma^2 - 4\gamma + 1)z^2 - (2 - 4\gamma)z + 2}{2(\gamma z+1)^2}. 
\end{equation}
 \end{subequations}
 
 It is easy to verify that $|\mathcal{R}(z)|\leq1 (\forall z\in\mathbb{C}^+)$, i.e., the SDIRK method \eqref{eq2.9} is  unconditionally  stable,  if and only if $\gamma\geq\frac{1}{4}$. Otherwise  the method is only conditionally stable.  Here,  we consider two SDIRK methods
  \begin{equation}\label{eq2.11}
\underbrace{\begin{array}{r|cc}
0.2&0.2 &0     \\
0.8&0.6  &0.2\\
\hline
  &\frac{1}{2}&\frac{1}{2}
  \end{array}}_{{{\rm conditionally~ stable, ~2nd-order}}},~~~~~~\underbrace{\begin{array}{r|cc}
\frac{3+\sqrt{3}}{6} &\frac{3+\sqrt{3}}{6} &0     \\
\frac{3-\sqrt{3}}{6}&-\frac{\sqrt{3}}{3}   &\frac{3+\sqrt{3}}{6}\\
\hline
  &\frac{1}{2}&\frac{1}{2}
  \end{array}}_{{\rm unconditionally stable, ~3rd-order}},~
\end{equation}
which correspond to    $\gamma=0.2$  and $\gamma=\frac{3+\sqrt{3}}{6}$ respectively (for both methods $b=\frac{1}{2}$).    Let $\Delta x=0.01$ and $\Delta t=0.02$. Then, for  two values of $\nu$: $\nu=10^{-3}$ and $\nu=2\times10^{-4}$,  we show in Figure \ref{fig2.1} the spectrum $\sigma(\Delta tA)$ and the stability region of the SDIRK method with $\gamma=0.2$, i.e., the second method in \eqref{eq2.11}.    We see that for $\nu=10^{-3}$ this  SDIRK method is stable, while it is not for $\nu=2\times10^{-4}$.   For   $\nu=2\times10^{-4}$, the second SDIRK method in \eqref{eq2.11} is suitable  because it is unconditionally stable.       By using the two  methods in \eqref{eq2.11} for $\nu=10^{-3}$ and $\nu=2\times10^{-4}$ respectively, we show  in Figure \ref{fig2.2}   the measured error $\|\bm{err}^k\|_{\infty}$ and $\|(I_t\otimes P)\bm{err}^k\|_{\infty}$, together with the upper bound predicted by $\frac{\alpha}{1-\alpha}$ (dotted line).   We see that   $\|\bm{err}^k\|_{\infty}$ and $\|(I_t\otimes P)\bm{err}^k\|_{\infty}$ decay with a very similar rate and that the  estimate  \eqref{eqn:err-single} predicts such a rate very well. Both $\|\bm{err}^k\|_{\infty}$ and $\|(I_t\otimes P)\bm{err}^k\|_{\infty}$ do not 
smoothly  decay to  the machine precision 2.2204$\times10^{-16}$ and at the last few iterations they stagnant at the level $O(10^{-13})$ for $\alpha=0.1$ and $O(10^{-12})$ for $\alpha=0.01$.  Such a stagnation is due to the roundoff error arising from the block diagonalization of the preconditioner $\mathcal{P}_\alpha$.  
   
     \begin{figure}[h!]
\begin{center}
		\includegraphics[height=2in,width=2.7in]{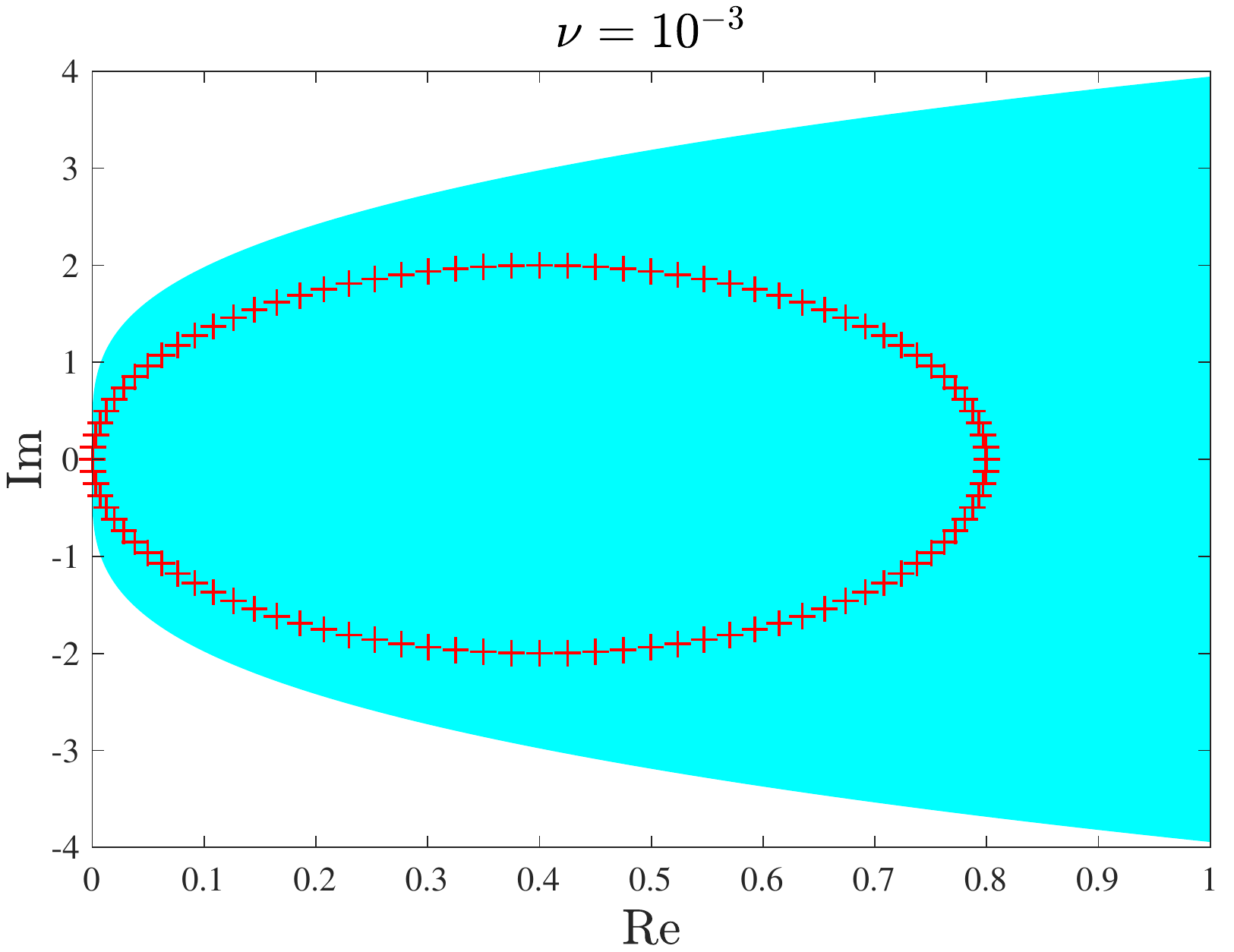}\includegraphics[height=2in,width=2.7in]{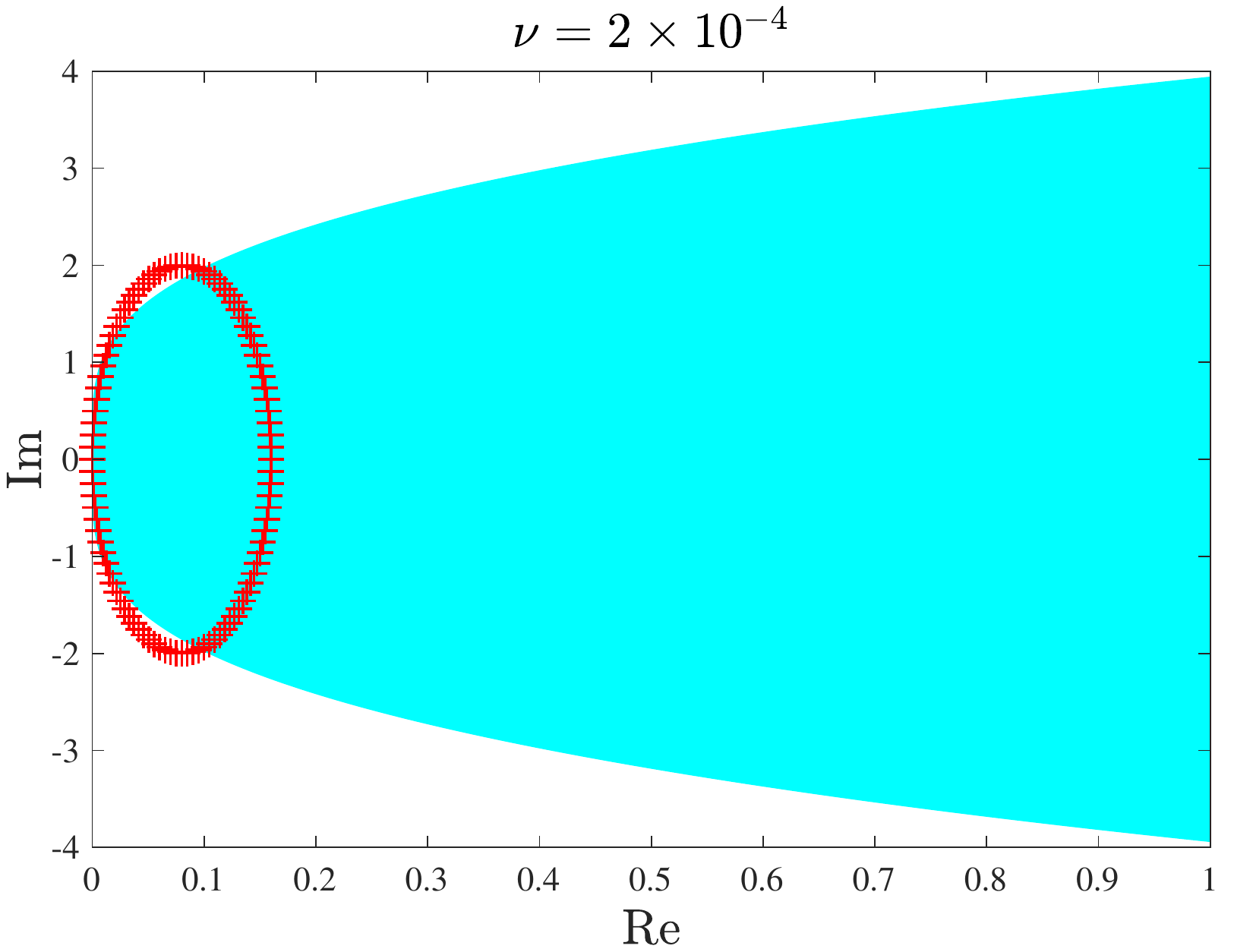}\\ 		\end{center}
	\caption{The stability region (shadow  region) of the second  SDIRK method  in \eqref{eq2.11} (i.e., $\gamma=0.2$)  and the spectrum $\sigma(\Delta tA)$ with $A$ being the matrix  given by \eqref{eq2.8} and $\Delta t=0.02$. } \label{fig2.1}
\end{figure} 

        \begin{figure}[h!]
\begin{center}
		\includegraphics[height=2in,width=2.7in]{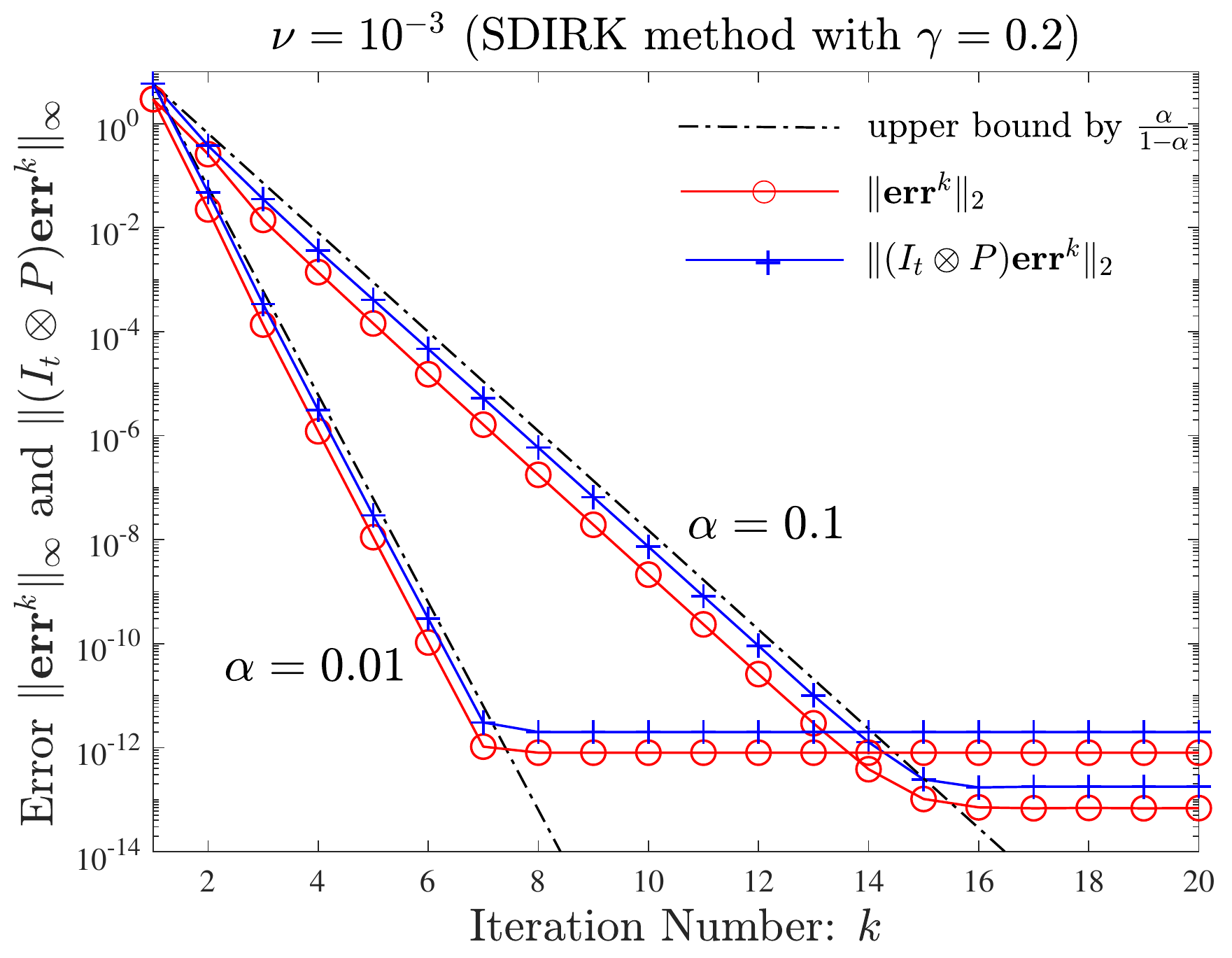}\includegraphics[height=2in,width=2.7in]{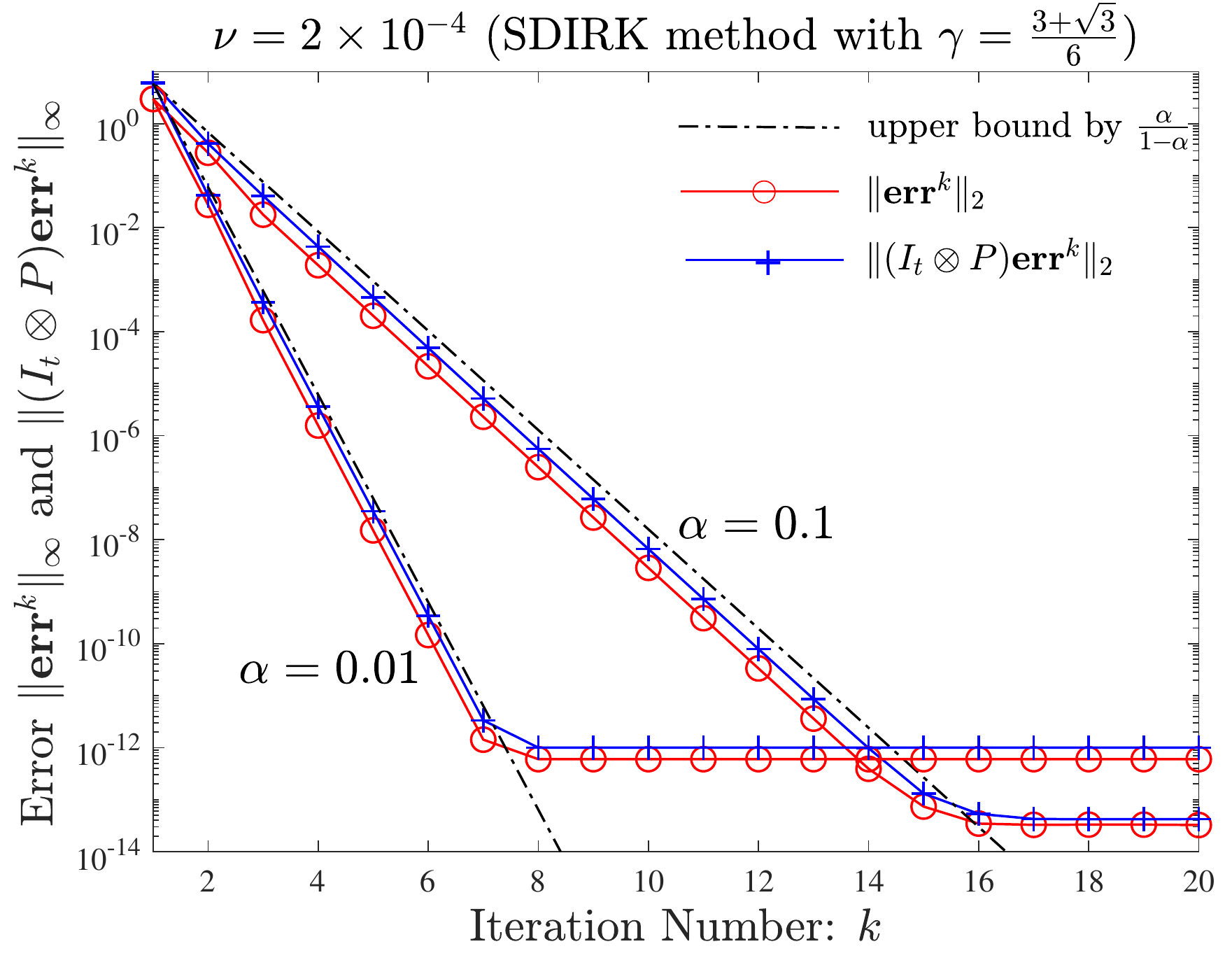}\\ 		\end{center}
	\caption{The measured error  $\|\bm{err}^k\|_{\infty}$ and $\|(I_t\otimes P)\bm{err}^k\|_{\infty}$ together with the upper bound by   $\frac{\alpha}{1-\alpha}$. Left: $\nu=10^{-3}$ and we use the SDIRK method in \eqref{eq2.11} with $\gamma=0.2$. Right:  $\nu=2\times10^{-4}$ and we use the SDIRK method in \eqref{eq2.11} with $\gamma=\frac{3+\sqrt{3}}{6}$. Here, $T=10$ and $\Delta t=0.02$.} \label{fig2.2}
\end{figure} 

We note that the advection-dominated diffusion  equation  is a well-known difficult  problem for PinT computation. In particular,  for equation \eqref{eq2.7} with $\nu=10^{-3}$ and $\nu=2\times10^{-4}$ the convergence rate of the  mainstream PinT algorithms  \textit{parareal} and \textit{MGRiT} is rather disappointing and the convergence rate continuously deteriorates as we refine $\Delta x$ (see \cite[Section 1]{GW20} for more details).  
\section{Convergence of iteration \eqref{eqn:precond-iteration}  for linear multistep methods}\label{sec:multistep}
In this section, we {consider} the  linear multistep methods.
For parabolic equations, the parallel algorithms for BDF methods (with initial corrections) up to order six
has been proposed and analyzed in \cite{WuZhou:BDF},
and has been generalized to the diffusion models with historical memory effects.
In the present work, we  propose a systematic
approach, showing the convergence of the preconditioned iteration, 
which works for all stable linear multistep methods.
 
We consider the $r$-step method ($r\ge1$) with
time levels $t_n=n\Delta t$ {($n=1,\ldots, N_t$)}: for given $y_0$, $y_1$, \ldots, $y_{r-1}$,
we find $y_{n+r}$ such that
\begin{equation}\label{eqn:multistep}
 {\sum}_{j=0}^r a_j y_{n+r-j} =  \Delta t {\sum}_{j=0}^r b_j (-A  y_{n+r-j}+ g_{n+r-j}), \quad n=0,1,\ldots, N_t-r.
\end{equation}
Without loss of generality, we assume that $a_0=1$. Then
we define the characteristic polynomials for $z\in \mathbb{C}^+$:
\begin{equation}\label{eqn:char-poly}
 p(s;z) = {\sum}_{j=0}^r a_j s^{r-j} + z b_j s^{r-j}.
\end{equation}
We assume that the $r$-step method is stable (Assumption 2).  
In principle, for any given angle $\theta\in(0,\pi/2)$, there exist $A(\theta)$-stable $r$-step methods of
order $r$ for every $r\ge1$.

\begin{assumption}\label{ass2}
The matrix $A$ in the linear system \eqref{eq1.1} is diagonalizable as $A=PD_AP^{-1}$ {with $\sigma(A)\subset\mathbb{C}^+$}.
Then for $z=\Delta t \lambda$ {with $\lambda\in\sigma(A)$}, the characteristic polynomial satisfies
the following root condition:
\begin{equation}\label{eqn:ass2}
p(s;z)=0~~\Longrightarrow~~
\begin{cases}
\text{either}~~ |s| < 1,\\
\text{or}~~|s|=1~\text{and it is a root of multiplicity 1}.
\end{cases}
 \end{equation}
\end{assumption}

Assumption \ref{ass2} immediately implies a stability estimate of the following discrete initial value problem:
for given $\xi_0$, $\xi_1$, \ldots, $\xi_{r-1}$, {find $\xi_{r+n}$  such that}
$$ \sum_{j=0}^r a_j \xi_{n+r-j} =  \sum_{j=0}^r b_j   (-z\xi_{n+r-j}+\Delta t f_{n+r-j}), \quad n=0,1,\ldots, N_t-r,
   $$
where   $z=\lambda \Delta t$  and $\lambda\in\sigma(A)$. Then Assumption \ref{ass2} leads to {the following}
  stability  {result}  \cite[Theorem 6.3.2]{Gautschi:book}
$$ |\xi_n|\le c \Big(\max_{0\le j\le r-1} |\xi_j|+\Delta t \sum_{j=r}^n |f_j|\Big),$$
where the constant $c$ is independent of $n$, $\Delta t$ and $N_t$.

Next, we describe the parallel-in-time solver of the  {time-stepping} scheme \eqref{eqn:multistep}. 
{Similar to the one-step methods, we can also  form   \eqref{eqn:multistep} into an}  {\sl{all-at-once}} {system} \eqref{eqn:all-at-once} with
\begin{equation}\label{eqn:aao-multistep}
  \mathcal{K} =B_{1}\otimes I_x+ \Delta t B_{2} \otimes A,
\end{equation}
where   $B_1, B_2\in\mathbb{R}^{N_t\times N_t}$ are Toeplitz matrices 
\begin{equation*} 
B_{1} = \begin{bmatrix}
a_0 & & & & &\\
a_1 &a_0 & & & &\\
\vdots &\ddots &\ddots & & &\\
a_{r} &   &\ddots &\ddots &  &\\
  &\ddots    &  &a_1 &a_0  &\\
  & &a_r &\dots &a_1 &a_0
\end{bmatrix},~ 
B_2 =\begin{bmatrix}
b_0 & & & & &\\
b_1 &b_0 & & & &\\
\vdots &\ddots &\ddots & & &\\
b_{r} &   &\ddots &\ddots &  &\\
  &\ddots    &  &b_1 &b_0  &\\
  & &b_r &\dots &b_1 &b_0
\end{bmatrix}.
\end{equation*}
 {Again, we  solve the all-at-once system \eqref{eqn:all-at-once} by the preconditioned iteration \eqref{eqn:precond-iteration} with  preconditioner}
\begin{equation}\label{eqn:P-multi}
\mathcal{P}_\alpha=C_1(\alpha)\otimes I_x+C_2(\alpha)\otimes A, 
 \end{equation}
 where $C_1(\alpha)$ and $C_2(\alpha)$  are  $\alpha$-circulant matrices
 \begin{equation*}
 C_1(\alpha)=\begin{bmatrix}
a_0 & &\alpha a_r &\dots &\alpha a_2 &\alpha a_1\\
a_1 &a_0 & & & &\alpha a_2\\
\vdots &\ddots &\ddots & &\ddots &\vdots\\
a_{r} &   &~\ddots &\ddots &  &\alpha a_r\\
  &\ddots    &  &a_1 &a_0  &\\
  & &a_r &\dots &a_1 &a_0
\end{bmatrix},~  C_2(\alpha)=\begin{bmatrix}
b_0 & &\alpha b_r &\dots &\alpha b_2 &\alpha b_1\\
b_1 &b_0 & & & &\alpha b_2\\
\vdots &\ddots &\ddots & &\ddots &\vdots\\
b_{r} &   &\ddots &\ddots &  &\alpha b_r\\
  &\ddots    &  &b_1 &~b_0  &\\
  & &b_r &\dots &b_1 &b_0
\end{bmatrix}.
 \end{equation*}

Then we are ready to show the convergence of the preconditioned iteration \eqref{eqn:precond-iteration}  {for the linear multistep methods}.
\begin{theorem}\label{thm:multi}
{Under   Assumption 2,  the error at the $k$-th iteration in} \eqref{eqn:precond-iteration}, denoted by ${\bm{err}}^{k}={\bm u}^{k}-{\bm u}$,
satisfies 
\begin{equation}\label{eqn:err-multi}
\|(I_t\otimes P){\bm{err}}^{k+1}\|_{\infty}\leq\frac{c\alpha}{1-c\alpha}\|(I_t\otimes P){\bm{err}}^{k}\|_{\infty},
\end{equation}
%provided the time-integrator \eqref{eq2.2} is stable.%, i.e., $\max_{z\in\sigma(\Delta tQ)}|\mathcal{R}(z)|\leq1$,
%where $\sigma(\Delta tQ)=\cup_{j=1,2,\dots, N_x}\{-{\rm i}\Delta t\sqrt{\lambda_j(A)}, {\rm i}\Delta t\sqrt{\lambda_j(A)}\}$ denotes the spectrum of $\Delta tQ$.
where $c>0$ denotes a generic constant independent of $\Delta t$ and $N_t$. 
Therefore, the iteration \eqref{eqn:precond-iteration} converges linearly if $\alpha\in(0,1/2c)$.
\end{theorem} 
 
\begin{proof} 
{From
\eqref{eqn:precond-iteration},  the error  ${\bm{err}}^{k}$ satisfies}
 \begin{equation*}
{\bm{err}}^{k+1}={\bm{err}}^{k}-(\mathcal{P}_\alpha^{-1}\mathcal{K}){\bm{err}}^{k}\qquad \forall~ k\ge0.
 \end{equation*}
We apply $I_t\otimes P$ on both sides of the above equation and define  
$$
{\bm \zeta}^{k}=(({\bm\xi}^{k}_{1})^\top, ({\bm \xi}^{k}_{2})^\top, \dots, ({\bm \xi}^{k}_{m})^\top)^\top=(I_t\otimes P){\bm{err}}^{k}.
$$ 
Let  $z$ be an arbitrary eigenvalue of $\Delta t A$,  
${\bm \xi}^{k}\in\mathbb{R}^{N_t}$ be the corresponding subvector of ${\bm \zeta}^{k}$.  Define
$$
K(z)=C_1+ zC_2,~P_\alpha(z)=C_1(\alpha)+ z C_2(\alpha).
$$
Then, it is clear that
   \begin{equation}\label{eqn:thm2-0}
{\bm \xi}^{k+1}={\bm \xi}^{k}-(P_\alpha^{-1}(z)K(z)){\bm \xi}^{k}, ~ \forall k\ge0. 
 \end{equation}
 
 By applying {$P_\alpha$} on both sides of \eqref{eqn:thm2-0}, we derive 
    \begin{equation}\label{eqn:thm2-1}
   \begin{split}
P_\alpha(z) {\bm \xi}^{k+1}=(P_\alpha(z) -K(z)){\bm \xi}^{k}.
\end{split}
 \end{equation}
Let  
${\bm \xi}^{k}=(\xi^{k}_{r}, \xi^{k}_{r+1}, \dots, \xi^{k}_{N_t})^\top.$
Then {a routine calculation yields} 
\begin{equation*} 
{\begin{split}
\left(C_1 + zC_2\right) {\bm \xi}^{k+1} = \alpha\begin{bmatrix}
\sum_{j=1}^r (a_j+zb_j)(\xi^{k+1}_{N_t-j+1} - \xi^{k}_{N_t-j+1})\\
\sum_{j=2}^r (a_j+zb_j)(\xi^{k+1}_{N_t-j+2} - \xi^{k}_{N_t-j+2})\\
\vdots\\
(a_r+zb_r)(\xi^{k+1}_{N_t} - \xi^{k}_{N_t})\\
\vdots\\
0
\end{bmatrix}.
\end{split}}
\end{equation*}
{Since $P_\alpha(z)-K(z)=C_1 + zC_2$},  \eqref{eqn:thm2-1} is equivalent to the following {difference equations} 
 \begin{equation}\label{eqn:thm2-2}
 \begin{cases}
  \sum_{j=0}^r a_j \xi_{n-j}^{k+1} =  -z \sum_{j=0}^r b_j  \xi_{n-j}^{k+1}, &n=r,r+1,\ldots, N_t-r,\\
\xi^{k+1}_{r-n}=\alpha (\xi^{k+1}_{N_t+1-n} - \xi^{k}_{N_t+1-n}),  &n=1,\ldots, r.
\end{cases}
\end{equation}
By Assumption \ref{ass2}, we have the stability estimate
\begin{equation}\label{eqn:thm2-3}
 \max_{0\le n \le N_t} |  \xi_{n}^{k+1}  | \le c_0 \max_{0\le n\le r-1}| \xi^{k+1}_{n}  | 
 \le c_0 \alpha \max_{0\le n\le r-1} |\xi^{k+1}_{N_t-n} - \xi^{k}_{N_t-n}|,
 \end{equation}
 with a generic constant $c_0>1$.
%=   c_0 \alpha |\xi^{k+1}_{N_t-j_0} - \xi^{k}_{N_t-j_0}|$$
This implies 
$$ \max_{0\le n\le r-1}  |\xi^{k+1}_{N_t-n}  |
\le  c_0 \alpha   \Big(\max_{0\le n\le r-1}|\xi^{k+1}_{N_t-n}| + \max_{0\le n\le r-1}  |\xi^{k}_{N_t-n}| \Big).$$
Rearranging terms in {the} above inequality, we arrive at
$$   \max_{0\le n\le r-1}  |\xi^{k+1}_{N_t-n}  |  \le \frac{c_0 \alpha }{1- c_0 \alpha } \max_{0\le n\le r-1} |  \xi^{k}_{N_t-n}|. $$
{Then, applying the stability estimate  \eqref{eqn:thm2-3} again leads to} 
\begin{equation}\label{eqn:thm2-4}
\begin{split}
 \max_{r\le n \le N_t- r} |  \xi_{n}^{k+1}  | %&\le c_0 \alpha \max_{0\le n\le r-1} |\xi^{k+1}_{N_t-n} - \xi^{k}_{N_t-n}|\\
& \le c_0 \alpha   \Big(\max_{0\le n\le r-1}|\xi^{k+1}_{N_t-n}| + \max_{0\le n\le r-1}  |\xi^{k}_{N_t-n}| \Big)\\
&\le c_0 \alpha   \Big(c_0 \max_{r\le n \le N_t- r} |  \xi_{n}^{k+1}  | +  c_0  \max_{r\le n\le N_t- r}  |\xi^{k}_{n}|  \Big).
\end{split}
 \end{equation}
We arrange the inequality and obtain
$$ \max_{r\le n \le N_t- r }  |\xi^{k+1}_{n}  | 
\le  \frac{c_0^2 \alpha}{1-c_0^2\alpha}    \max_{r\le n\le N_t-r}  |\xi^{k}_{n}|.$$
Letting $c=c_0^2$, we derive
$$ \max_{r\le n \le N_t}  |\xi^{k+1}_{n}  | 
\le  \frac{c\alpha}{1-c \alpha}    \max_{r\le n \le N_t}   |\xi^{k}_{n}|,
$$
which implies that 
\begin{equation}\label{eq3.14}
  \|{\bm \xi}^{k+1} \|_\infty
\le  \frac{c\alpha}{1-c \alpha}      \|{\bm \xi}^{k} \|_\infty.
\end{equation}
Now, substituting ${\bm \zeta}^{k}=(I_t\otimes P){\bm{err}}^{k}$    into  \eqref{eq3.14} gives the desired result \eqref{eqn:err-multi}. 
\end{proof} 
 
\begin{remark}
For the multistep method, we note that the convergence factor involves a generic constant $c\geq1$.
It depends on the stability of the time stepping scheme, but it is always independent of the 
step size $\Delta t$, the time level $n$ and the total number of steps $N_t$. 
\end{remark} 
  
To complete this section, we test two four-step multistep methods, the four-step BDF method (BDF4) \cite{WuZhou:BDF} and the modified 
 Adams-Moulton method (AM) \cite{ChenWang:2016}:
 \begin{equation}\label{eq3.15}
 \begin{split}
 &({\rm BDF4})~~~~y_{n+1}-\frac{48}{25}y_{n}+\frac{36}{25}y_{n-1}-\frac{16}{25}y_{n-2}+\frac{3}{25}y_{n-3}+\frac{12\Delta tA}{25}y_{n+1}=0, \\
  &({\rm AM})~~~~\quad y_{n+1} - y_n + \Delta t A\left (\frac{2}{3} y_{n+1} + \frac{5}{12} y_{n-1} - \frac{1}{12} y_{n-3}\right)=0. \\
 \end{split}
 \end{equation}
 We apply the above multistep methods to the semi-discrete system of the advection-diffusion equation \eqref{eq2.7} with $T=8$, $\nu=10^{-3}$
 and $\Delta t=\Delta x=\frac{1}{128}$.  
 
 In Figure \ref{fig3.1}, we plot the stability regions of the above two multistep methods (shadow regions)
  together with  $\sigma(\Delta tA)$,  the spectrum  of $\Delta tA$.   We see that  both methods are stable for the semi-discrete system.

 In Figure \ref{fig3.2}, we plot  the measured error $\|\bm{err}^k\|_{\infty}$ and $\|(I_t\otimes P)\bm{err}^k\|_{\infty}$, together with the upper bound predicted by $\frac{\alpha}{1-\alpha}$ (dotted line).   We see that for AM scheme,    $\|\bm{err}^k\|_{\infty}$ and $\|(I_t\otimes P)\bm{err}^k\|_{\infty}$ 
 decay with a very similar rate close to $\alpha/(1-\alpha)$.  
 While for BDF4, as we see in Figure \ref{fig3.2} on the right,  
%  the  estimate  \eqref{eqn:err-multi}  with $c=1$ does no correctly predicts the 
%  decay of the error. The  error for  the first 6 iterations are given in Figure \ref{fig3.3},   where we see that 
 the decay rate is not  uniform {(see 
  the error of the first 6 iterations in Figure \ref{fig3.3})}. In particular, for $\alpha=0.1$  we have 
 $$
 \frac{\|(I_t\otimes P)\bm{err}^2\|_{\infty}}{\|(I_t\otimes P)\bm{err}^1\|_{\infty}}=0.54,~ \frac{\|(I_t\otimes P)\bm{err}^{k+1}\|_{\infty}}{\|(I_t\otimes P)\bm{err}^k\|_{\infty}}\approx0.11\approx\frac{\alpha}{1-\alpha}~~ (\text{for} ~k\ge1), 
 $$
 and for $\alpha=0.01$ we have 
  $$
 \frac{\|(I_t\otimes P)\bm{err}^2\|_{\infty}}{\|(I_t\otimes P)\bm{err}^1\|_{\infty}}=0.05,~ \frac{\|(I_t\otimes P)\bm{err}^{k+1}\|_{\infty}}{\|(I_t\otimes P)\bm{err}^k\|_{\infty}}
 \approx0.01  \approx\frac{\alpha}{1-\alpha}~~ (\text{for} ~k\ge1){\color{red}.}
 $$
 This contrasts sharp with the  {one}-step methods, where the convergence factor is alway bounded (from above) by $\alpha/(1-\alpha)$
 as suggested in Theorem \ref{thm:single}.
 However, the numerical results suggest that the convergence rate will approaches $\alpha/(1-\alpha)$ from the second iteration.
 This interesting phenomenon warrants further investigation.
 
 Similar to Figure \ref{fig2.2},  
 $\|\bm{err}^k\|_{\infty}$ and $\|(I_t\otimes P)\bm{err}^k\|_{\infty}$ do not 
  decay to  the machine precision  and at the last few iterations the error  stagnants at certain  level, due to the roundoff error arising in the block diagonalization of the preconditioner $\mathcal{P}_\alpha$.  

      \begin{figure}[h!]
\begin{center}
		\includegraphics[height=1.9in,width=2.5in]{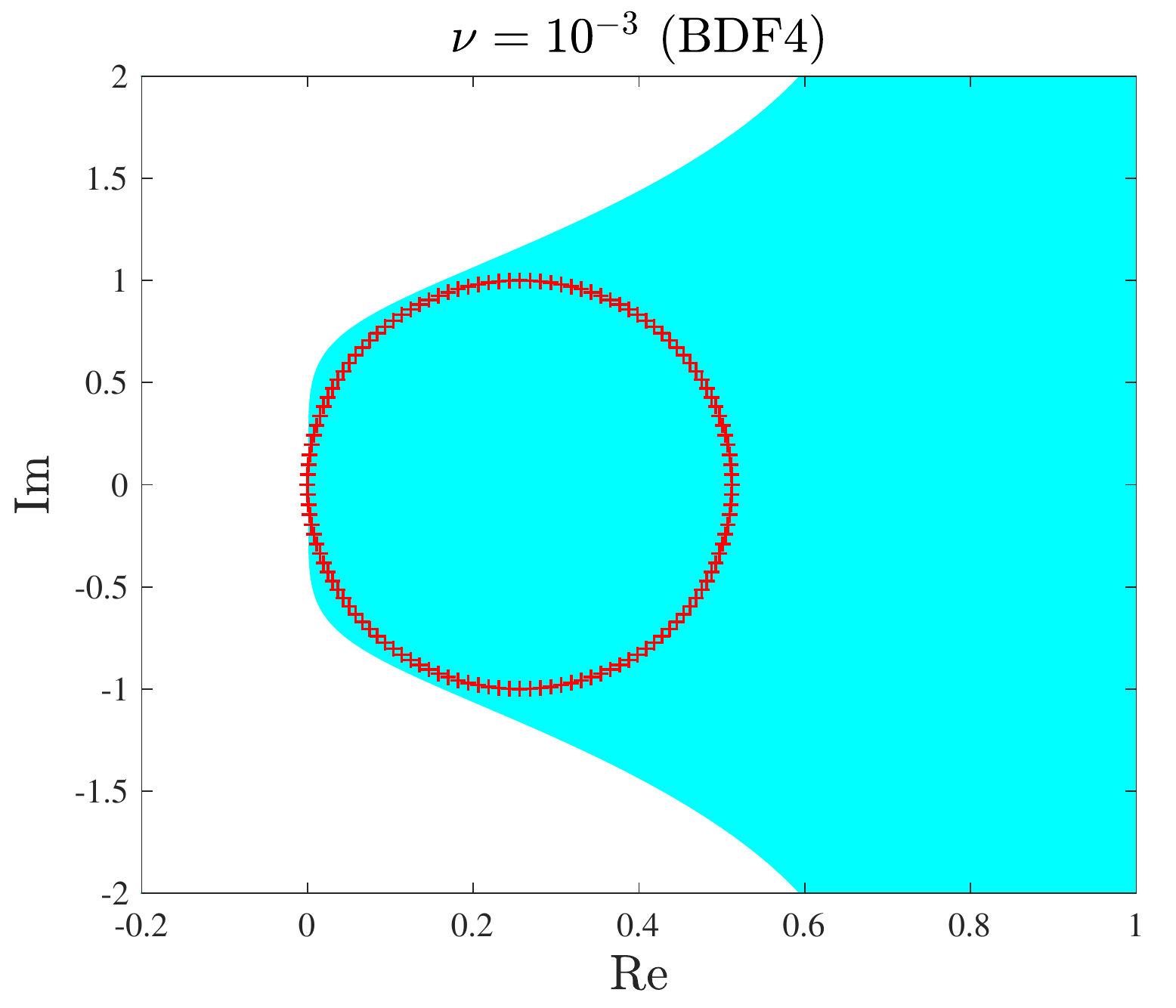}\includegraphics[height=1.9in,width=2.5in]{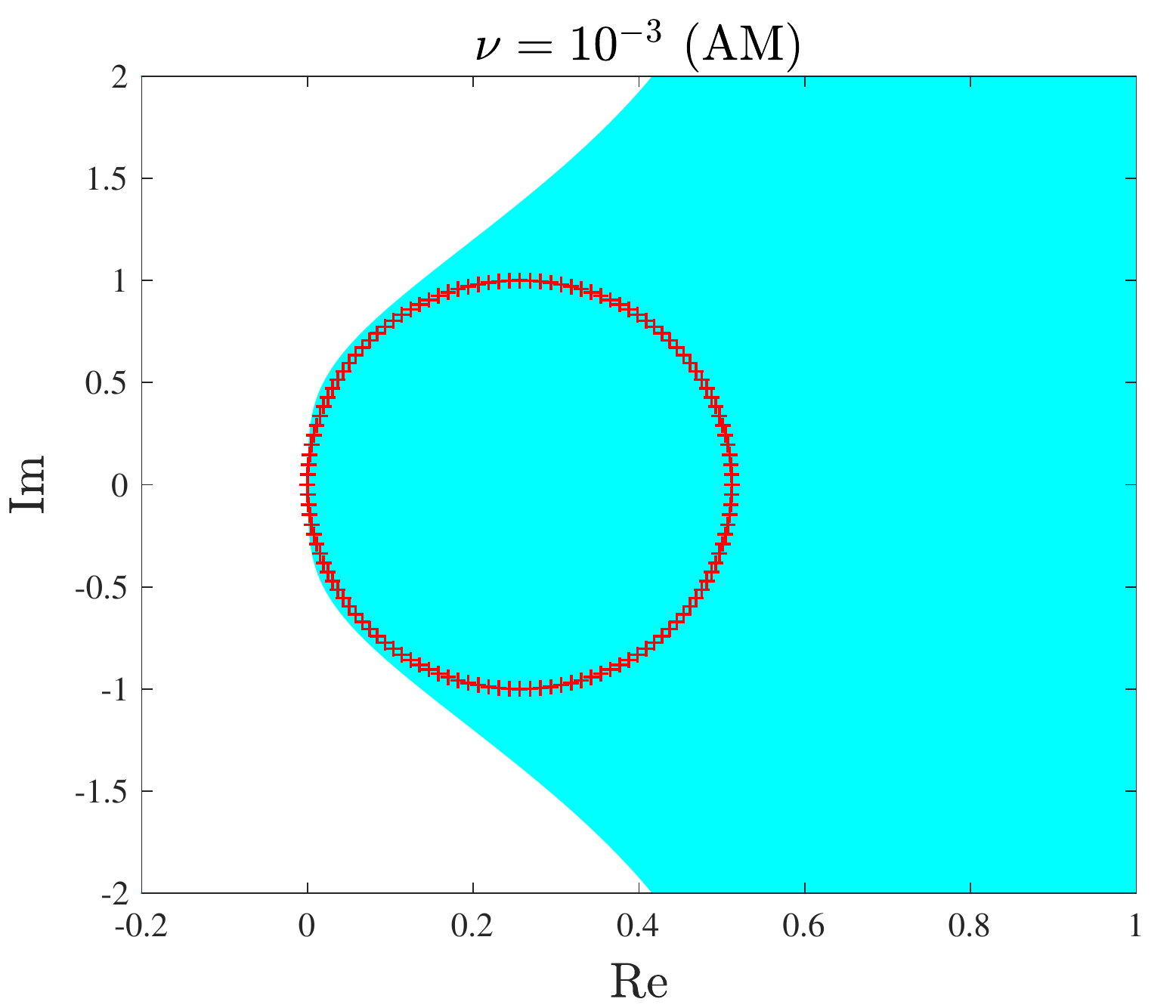}\\ 		\end{center}
	\caption{The stability region (shadow  region) of the BDF4 method  (left) and the AM method (right) in \eqref{eq3.15} and the spectrum $\sigma(\Delta tA)$ with $A$ being the matrix  given by \eqref{eq2.8} and $\Delta t=\Delta x=\frac{1}{128}$ and $\nu=10^{-3}$. } \label{fig3.1}
\end{figure}
 
        \begin{figure}[h!]
\begin{center}
		\includegraphics[height=1.9in,width=2.5in]{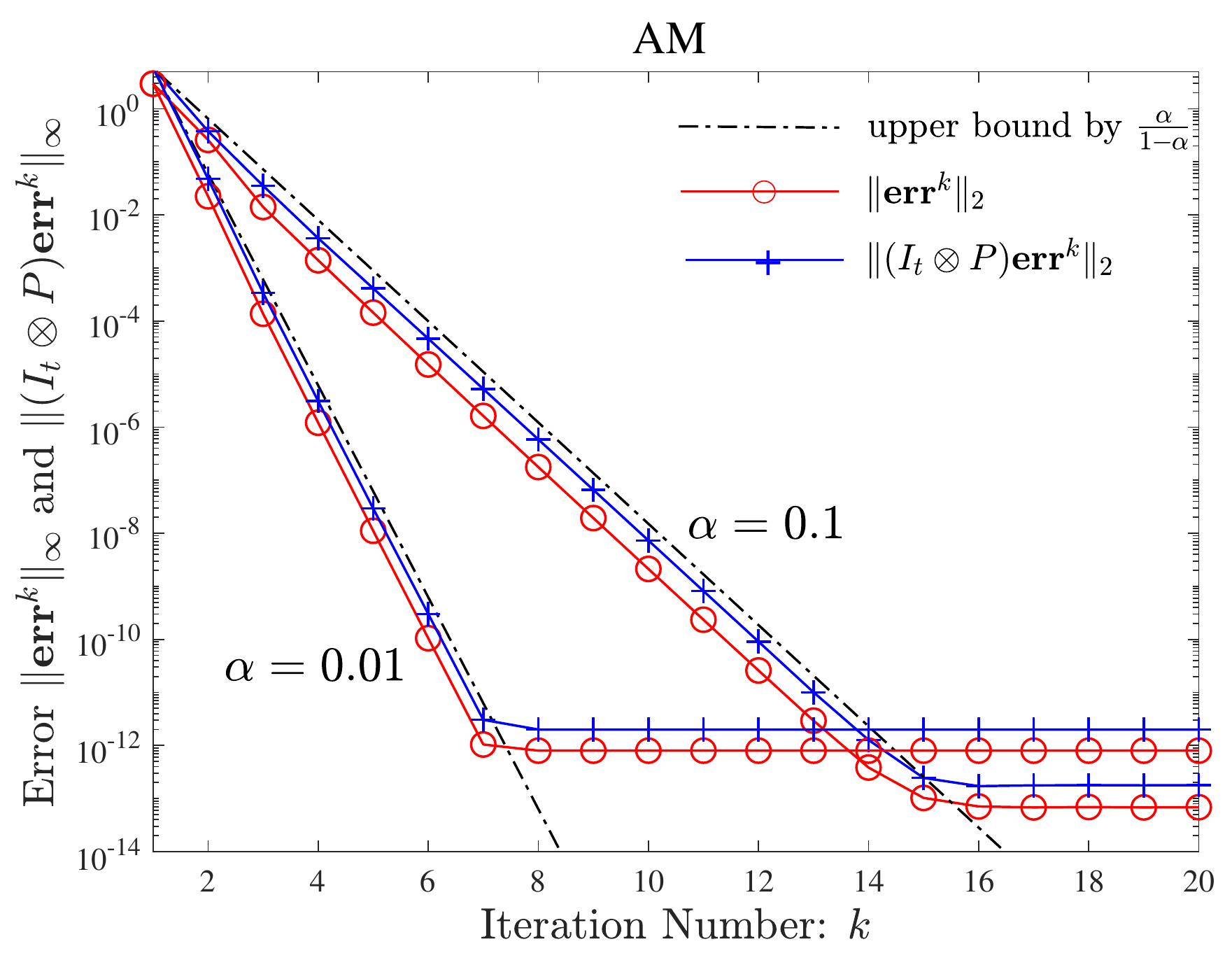}\includegraphics[height=1.9in,width=2.5in]{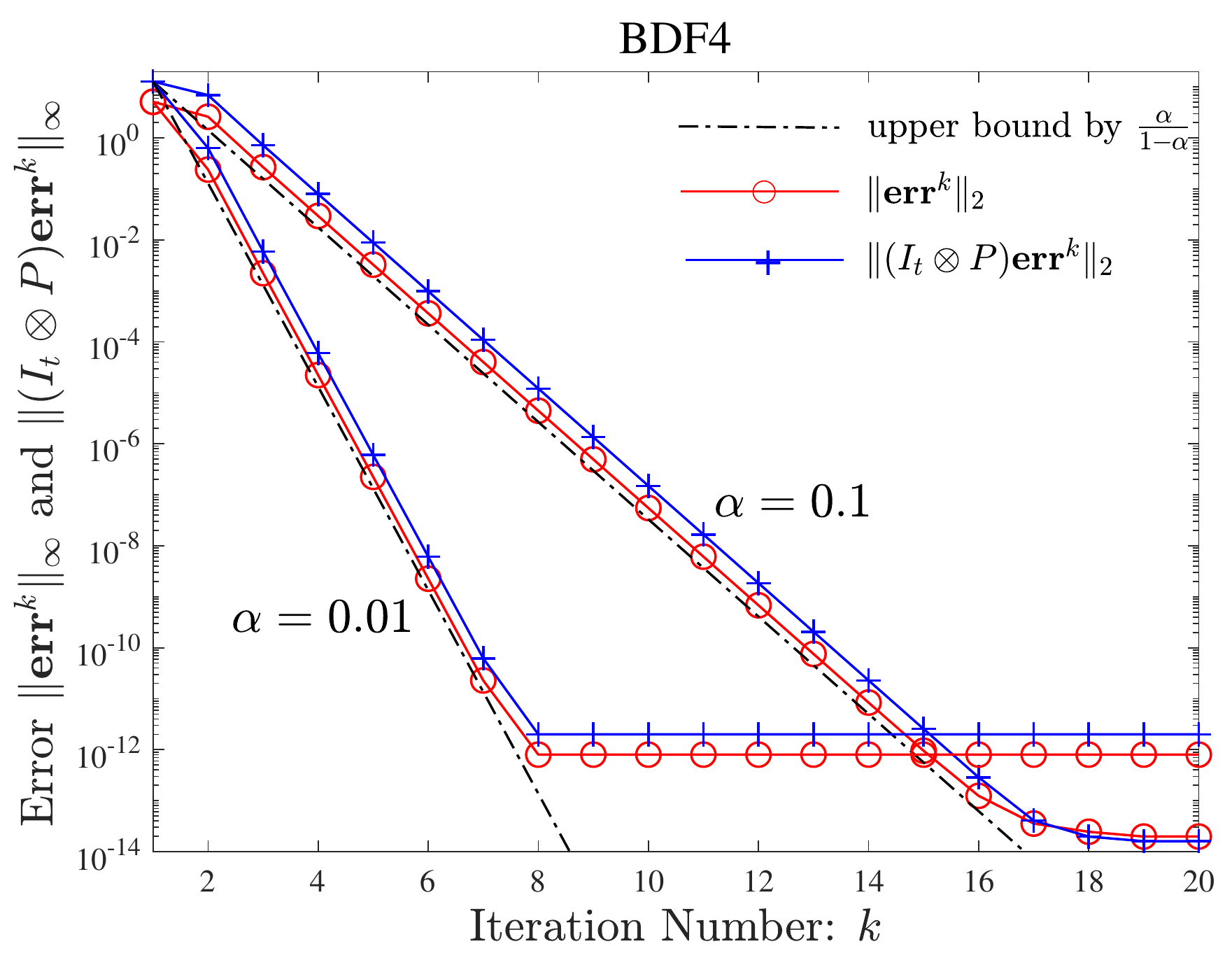}\\ 		\end{center}
	\caption{Applying the preconditioned iteration method \eqref{eqn:precond-iteration} to the advection-diffusion equation \eqref{eq2.7}, the measured error  $\|\bm{err}^k\|_{\infty}$ and $\|(I_t\otimes P)\bm{err}^k\|_{\infty}$ together with the upper bound by   $\frac{\alpha}{1-\alpha}$. Left: the 4-step Adams method in \eqref{eq3.15}. Right:  the 4-step BDF method in \eqref{eq3.15}. Here, $T=8$, $\Delta t=\Delta x=\frac{1}{128}$ and $\nu=10^{-3}$. } \label{fig3.2}
\end{figure} 

        \begin{figure}[h!]
\begin{center}
		\includegraphics[height=1.9in,width=2.5in]{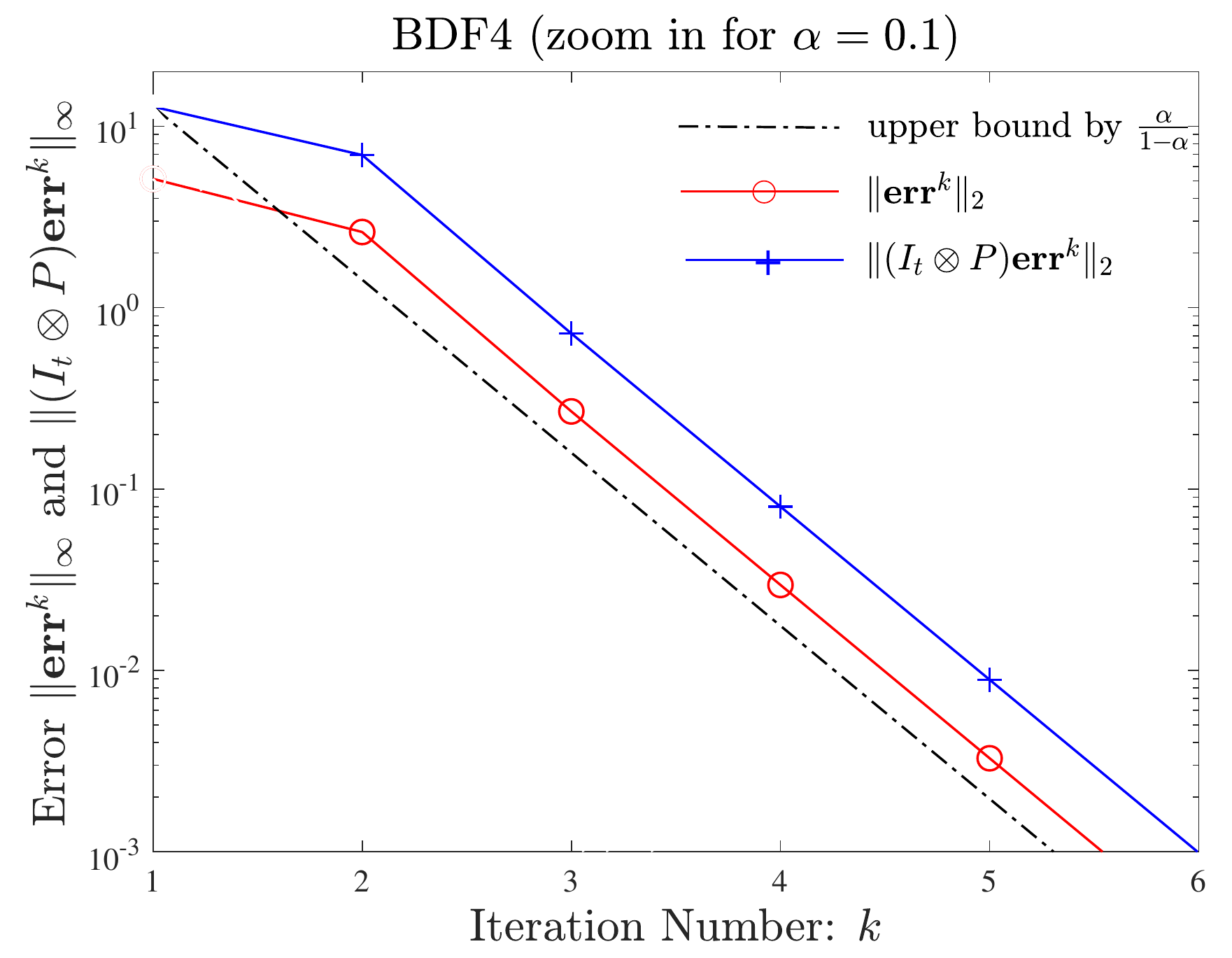}\includegraphics[height=1.9in,width=2.5in]{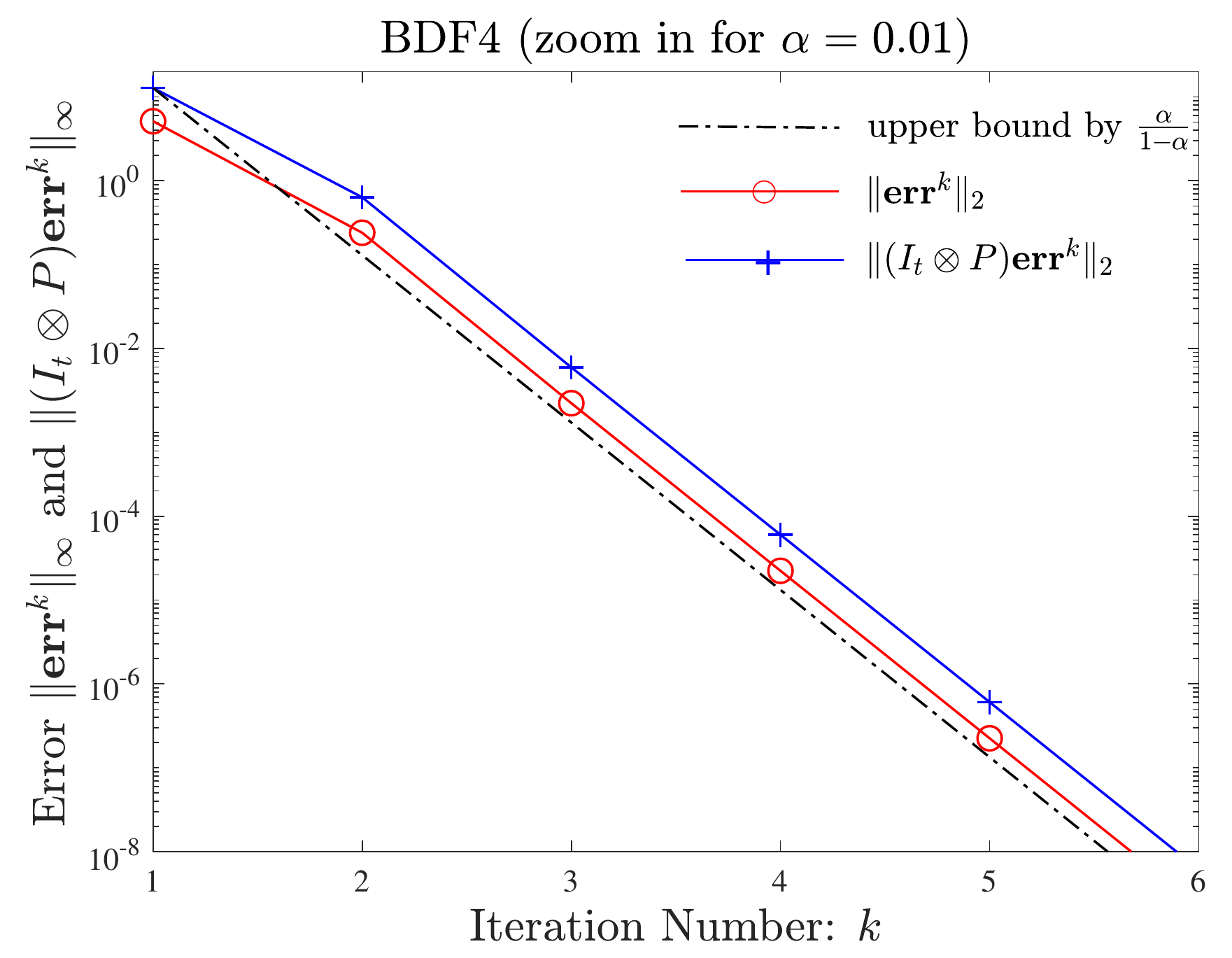}\\ 		\end{center}
	\caption{For  Figure \ref{fig3.2} on the right, the local details for the first 6 iterations. Left: $\alpha=0.1$. Right: $\alpha=0.01$.  } \label{fig3.3}
\end{figure}

 \section{Conclusion and discussion}\label{sec4}
 We   gave a uniformed proof for a class of PinT algorithm, which is  based on formulating the time discretization into an all-at-once system $\mathcal{K}{\bm u}={\bm b}$ and then solving this system iteratively by  using a block $\alpha$-circulant matrix  $\mathcal{P}_\alpha$ as the preconditioner. Existing work focus on examining  the spectrum of the iterative  matrix $\mathcal{I}-\mathcal{P}_\alpha^{-1}\mathcal{K}$ and this leads to many case-by-case  studies depending on the underlying time-integrator.  In this short paper, we proved that the error of the algorithm decays with a rate {only} depending on the parameter $\alpha$ for all {\sl{stable}} one-step methods. We also considered the linear multistep methods, for which we proved similar result provided the method is stable as well. The error decaying rate is slightly different from that of the one-step methods and in practical computation we indeed observed such a difference.  The proof given in this paper is completely different from existing work.  It  lies in  representing  the preconditioned  iteration as a special difference equations with head-tail coupled condition and relays on the stability function only.  
 
An important work  that we do not concern  is how to get  the increment $\Delta {\bm u}^k=\mathcal{P}_\alpha^{-1}{\bm r}^k$  in \eqref{eqn:precond-iteration}. For  one-step methods,  we can factorize $\mathcal{P}_\alpha$  in \eqref{eqn:P} as
$$
\mathcal{P}_\alpha=(V\otimes I_x)(I_t\otimes I_x+D\otimes \mathcal{R}(\Delta tA))(V^{-1}\otimes I_x), 
$$
where $V$ and $D$ are  the eigenvector and eigenvalue matrices of the circulant matrix $C$, i.e., $C=VDV^{-1}$. Then, we can compute $\Delta{\bm u}^k$ via the following   steps 
\begin{equation}\label{eq4.1}
\begin{cases}
  {\bm p}=(V^{-1}\otimes I_x){\bm r}^k, &\text{Step-(a)}\\
 \left(I_x+\lambda_n\mathcal{R}(\Delta tA)\right)q_n=p_n, ~n=1,2,\dots, N_t, &\text{Step-(b)}\\
 \Delta {\bm u}^k=(V\otimes I_x){\bm q}, &\text{Step-(c)}\\
\end{cases}
\end{equation}
where ${\bm p}=(p_1^\top,  \dots, p_{N_t}^\top)^\top$ and ${\bm q}=(q_1^\top,   \dots, q_{N_t}^\top)^\top$.  For     Step-(a) and Step-(c), we only need to do matrix-vector multiplications and  the major computation burden is  Step-(b). These are completely decoupled linear systems and can be solved in parallel, but there are still many concrete issues that need to be explored.      

The first issue is about the general implicit RK method of stage $s$  specified
by the Butcher tableau
$$
 \begin{array}{r|l}
c &{{\Theta}}    \\
\hline
&b^\top
\end{array},
$$
  for which  $\mathcal{R}(\Delta tA)= I_x-b^\top\otimes (\Delta tA)(I_s\otimes I_x+\Theta\otimes(\Delta tA))^{-1}({\bm 1}\otimes I_x)$, where $I_s\in\mathbb{R}^{s\times s}$ is an identity matrix and ${\bm 1}=(1,1,\dots, 1)^\top\in\mathbb{R}^s$.  Then how to efficiently solve the linear system in  Step-(b) is an very important topic. Some preliminary experiences are as follows. For the two-stage RK methods, by an \textit{elimination} operation for the stage variables   it can be shown that  such a linear system is equivalent to   the following 
\begin{equation}\label{eq4.2}
  (a_0I_x +a_1(\Delta tA)+a_2(\Delta tA)^2)\tilde{p}=\tilde{q}, 
  \end{equation}
  where $a_{0}$, $a_1$, $a_2$ {depend}  on $\lambda_n$ and  the RK   method. The quadratic term $(\Delta tA)^2$ will be a serious problem for both the memory  storage  and the computation cost. The approach for this issue is to notice that equation \eqref{eq4.2} can be represented as 
  \begin{equation}\label{eq4.3}
  \left(a_0\underbrace{\begin{bmatrix}1 &\\ &1\end{bmatrix}}_{=I_s}\otimes I_x +\underbrace{\begin{bmatrix}a_1-\mu &\mu\\ a_1-\mu-\frac{a_0a_2}{\mu} &\mu\end{bmatrix}}_{:=W_s}\otimes(\Delta tA)\right)\begin{bmatrix}\tilde{p}\\ p_\dag\end{bmatrix}=\begin{bmatrix}\tilde{q}\\  \tilde{q}\end{bmatrix}, 
  \end{equation}
  where $p_\dag$ is an auxiliary variable and $\mu\neq0$ is a free parameter. Then, we can factorize  $W_s$ as $W_s=V_sD_sV_s^{-1}$ and solve \eqref{eq4.3} by using the diagonalization technique again. In this way, we only need to solve two independent linear systems of the form 
  \begin{equation}\label{eq4.4}
(\eta I_x+\Delta tA)v=b,
  \end{equation}
  where $\eta=a_0/D_s(1,1)$ (or  $\eta=a_0/D_s(2,2)$). (Here we assume $D_s(1,1)\neq0$, since otherwise  it is trivial to treat the  linear system.)  The above idea  was used in \cite{WZ19} for the  two-stage   Lobatto IIIC method  and  in \cite{WZ21} for the two-stage SDIRK method.  Similar to \eqref{eq4.2},  for a general $s$-stage RK method we believe that  the linear system in Step-(b) of  \eqref{eq4.1}  can be equivalently represented as 
  $$
  (a_0I_x+a_1(\Delta tA)+\cdots+a_{s}(\Delta tA)^s)\tilde{p}=\tilde{q}.
  $$
However, it is unclear how to further represent it as  \eqref{eq4.3} with a suitable matrix $W_s$. Such a $W_s$ should be diagonalizable with small    condition number of the eigenvector matrix $V_s$. 
  
 The second issue is about the computation of the complex-shift problem \eqref{eq4.4}. It seems this is a very fundamental problem   arising in various fields, such as the Laplace inversion approach, preconditioning technique for the Helmholtz equations and the   diagonalization technique here. {For ODEs with first-order  temporal  derivative},  in most cases we found that the quantity $\eta$ arising from the diagonalization technique has non-negative real part, i.e.,  $\Re(\eta)\geq0$,  and thus  it is not  difficult to solve \eqref{eq4.4}. This issue was carefully addressed for the multigrid method in \cite{MacLachlan:2008, Notay:2010}. 
 %\ZZ{maybe cite one classical paper working on multigrid solver for complex-valued poisson problems???}. 
 For second-order ODEs, it happens  $\Re(\eta)<0$ for some $\lambda_n$ in \eqref{eq4.1} and therefor \eqref{eq4.4} is a Helmholtz-like problem, a typical  hard   problem in numerics  when $|\Re(\eta)|$ is large and  the ratio $r:=|\Im(\eta)/\Re(\eta)|\ll1$. Cocquet and   Gander   made an interesting analysis for the multigrid method in  \cite{CG17} and the main conclusion is that the multigrid method always converges if $|\Im(\eta)|=\mathcal{O}(|\Re(\eta)|)$. However,  numerical results indicate that the  multigrid method could converge    arbitrarily slow and the convergence rate actually heavily depends on the ratio $r$.   A local Fourier analysis will reveal how the convergence rate depends on $r$ and this is one of our ongoing work.   Besides the multigrid method, there are also other work concerning   the complex-shift problem \eqref{eq4.4}  and to name a few we mention   the preconditioned GMRES method \cite{GG15}, the domain decomposition method \cite{GS17}   and the approach  based on   representing the real and imaginary parts separately,   
 leading to a  real symmetric linear system, which can be solved 
 by an  Uzawa-type method (see \cite[Section 5]{S14} for this approach). 
 
\bigskip
%\section*{Acknowledgements}
%The research of S. Wu is supported by  NSF of China (No. 11771313). The work of T. Zhou is supported by NSF of China (under grant numbers 11822111, 11688101and 11731006), the science challenge project (No. TZ2018001), and youth innovation promotion association (CAS). 
%The research of Z. Zhou is partially supported by the Hong Kong RGC grant (project no. 15304420). 
% 

\end{document}